\documentclass[12pt]{amsart}

\usepackage{amsfonts, amstext, amsmath, amsthm, amscd, amssymb}
\usepackage{epsfig, psfrag, color}

\let\oldmarginpar\marginpar
\renewcommand\marginpar[1]{\oldmarginpar[\raggedleft\footnotesize #1]%
{\raggedright\footnotesize #1}}

\theoremstyle{plain}
\newtheorem{theorem}{Theorem}[section]
\newtheorem{corollary}[theorem]{Corollary}
\newtheorem{lemma}[theorem]{Lemma}
\newtheorem{prop}[theorem]{Proposition}

\newtheorem*{no-num-theorem}{Theorem}

\theoremstyle{definition}
\newtheorem{define}[theorem]{Definition}
\newtheorem{remark}[theorem]{Remark}

\newcommand{\GG}{{\mathbb{G}}}
\newcommand{\cut}{{\backslash \backslash}}
\newcommand{\guts}{{\rm{guts}}}
\newcommand{\vol}{{\rm{vol}}}

\begin{document} 

\title[Volume and geometry of homogeneously adequate knots]{Volume and geometry of \\homogeneously adequate knots} 
\author[P.~Bartholomew]{Paige Bartholomew} 
\author[S.~McQuarrie]{Shane McQuarrie}
\author[J.~Purcell]{Jessica S. Purcell}
\author[K.~Weser]{Kai Weser}

\address{Mathematics Department, Brigham Young University, Provo, UT 84602, USA}


\begin{abstract}
We bound the hyperbolic volumes of a large class of knots and links, called homogeneously adequate knots and links, in terms of their diagrams.  To do so, we use the decomposition of these links into ideal polyhedra, developed by Futer, Kalfagianni, and Purcell.  We identify essential product disks in these polyhedra.
\end{abstract}

\maketitle 

\section{Introduction}\label{sec:intro}

A knot or link in $S^3$ is often represented by a diagram: a planar 4--valent graph with over--under crossing information.  Results from hyperbolic geometry tell us that the link complement is often hyperbolic \cite{thurston:bulletin}, and invariants of the hyperbolic geometry, such as the volume, are knot and link invariants \cite{mostow, prasad}. 
%
However, it is a nontrivial problem to bound the volume of a hyperbolic knot in terms of its diagram.  To date, this has been done for many classes of knots.
%
Lackenby found conditions that bound the volume of alternating links based on their diagrams \cite{lackenby:volume-alt}.  Purcell \cite{purcell:volumes}, and Futer, Kalfagianni, and Purcell found volume bounds of many additional classes of links \cite{fkp:filling, fkp:conway, fkp:coils}.  Most recently, they found bounds on the volumes of a class of links called semi-adequate links (defined in Definition~\ref{def:adequate}) in terms of their diagrams \cite{fkp}.  These links include alternating links, and many additional classes of links.

Much of the work in \cite{fkp} was extended to an even larger class of links, called homogeneously adequate links.  However, the step of bounding volumes in terms of diagrams was not completed in that paper.  It was shown that volumes could be bounded below by the Euler characteristic of a graph coming directly from the diagram, but that a certain quantity must be subtracted from this result.  This quantity, called $||E_c||$ in that paper, was shown to have relations to diagrams for semi-adequate links, but the connection was not generalized to homogeneously adequate links.

In this paper, we investigate this quantity for homogeneously adequate links.  We prove that it is bounded in terms of properties of the diagram.  Thus we can extend many results on volumes from \cite{fkp} to homogeneously adequate links.

Why homogeneously adequate links?  These links, defined carefully below, are known to include nearly all classes of knots and links mentioned above.  In particular, they include alternating links, positive and negative closed braids, closed 3--braids, Montesinos links, and homogeneous links.  They include links that are known not to be semi-adequate, including the closed 5--braid shown in Figure~\ref{fig:ex-homog} in this paper.  See \cite[Remark~1.5]{fkp:qsf}.  
At the time of this writing, it is unknown whether \emph{every} hyperbolic link admits a homogeneously adequate diagram.  See \cite{ozawa} for a detailed discussion.  


\subsection{Organization}

In Section~\ref{sec:knots}, we review the definition of homogeneously adequate links, and the polyhedral decomposition obtained for these links in \cite{fkp}.  In Section~\ref{sec:tentacles}, we review some of the results of \cite{fkp} that we will need in this paper, including a few of the arguments using a combinatorial proof technique that we call ``tentacle chasing''.  We prove the main technical result in Section~\ref{sec:EPDs}.  This leads to bounds on the quantity $||E_c||$ in terms of properties of the diagram.  The results of that section are applied to volumes in Section~\ref{sec:applications}.  

\subsection{Acknowledgements}

This work was partially supported by NSF grant DMS--1252687, and by funding for student research from Brigham Young University's College of Physical and Mathematical Sciences.  


\section{Properties of Knots}\label{sec:knots}

In this section, we review the definitions of semi-adequate and homogeneously adequate knots and links.  We also describe the polyhedral decomposition developed in Chapters 2 and 3 of \cite{fkp}.  In order to do so, we need to review a few other definitions.  Thus this section is particularly heavy with definitions and notation.  We attempt to include enough examples to make this paper as self--contained and readable as possible.  However, in the interest of space and readability, we must omit many of the details in \cite{fkp}, including the geometric ideas that led to the development of these polyhedra.  For more information on these ideas, and more information on the polyhedra and their other uses, see, for example, \cite{fkp:survey}, in addition to \cite{fkp}.

Given a crossing of a knot, there are two ways to resolve the crossing, called the $A$--resolution and the $B$--resolution.  See Figure \ref{fig:resolutions}.

\begin{figure}
\includegraphics{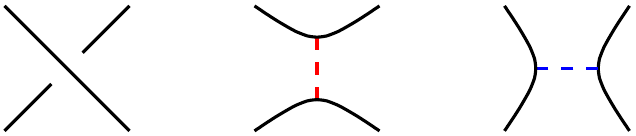}
\caption{From left to right: an unresolved crossing, the $A$--resolution, the $B$--resolution.}
\label{fig:resolutions}
\end{figure}

\begin{define}\label{def:state}
A choice of $A$-- or $B$--resolution for each crossing of the diagram is a \emph{state}, which we will denote by $\sigma$.  After applying $\sigma$ to the diagram, we obtain a crossing--free diagram consisting of a collection of closed curves called \emph{state circles}, which we denote by $s_\sigma$.  In Figure~\ref{fig:resolutions}, there is a dashed edge running between state circles in each resolution.  We will add these edges to the state circles in $s_\sigma$, obtaining a 3--valent graph, which we denote by $H_\sigma$.  The edges added to go from $s_\sigma$ to $H_\sigma$ we will call \emph{segments}.
\end{define}

\begin{define}\label{def:StateGraph}
The \emph{state graph} $\GG_\sigma$ is obtained from $H_\sigma$ by collapsing each state circle to a vertex.  Hence edges of $\GG_\sigma$ are segments of $H_\sigma$.  The \emph{reduced state graph} $\GG_\sigma'$ is obtained from $\GG_\sigma$ by removing all multiple edges between pairs of vertices.
\end{define}

For an example of a diagram and graphs $H_\sigma$, $\GG_\sigma$, and $\GG_\sigma'$ for the case $\sigma$ is the all--$A$ state, see Figure~\ref{fig:ex-state}, part of which appeared in \cite{fkp:survey}.  For an example of a diagram and $H_\sigma$ for a more general state, see Figure~\ref{fig:ex-homog}.

\begin{figure}
  \includegraphics{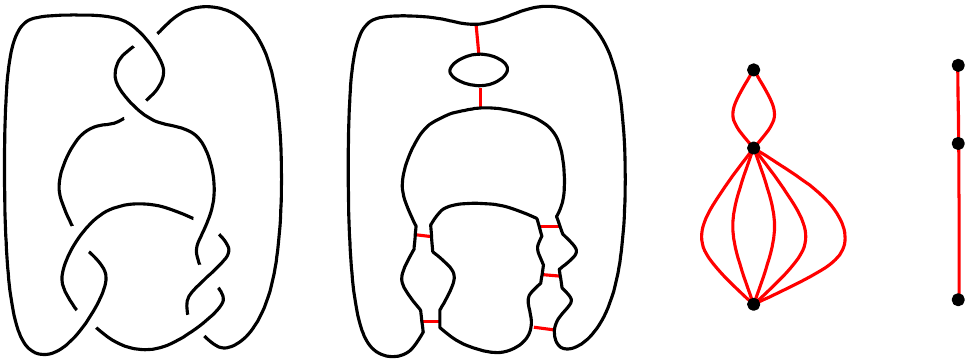}
  \hspace{.2in}
  \includegraphics{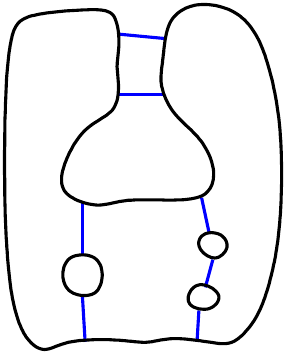}
  \caption{Left to right: A diagram, the graph $H_A$ (i.e.\ the state $\sigma$ is the all--$A$ state), the graph $\GG_A$, the graph $\GG_A'$, and the graph $H_B$ (i.e.\ the state is the all--$B$ state)}
  \label{fig:ex-state}
\end{figure}


\begin{define}\label{def:adequate}
A diagram $D(K)$ is called \emph{$\sigma$--adequate} if $\GG_\sigma$ has no one--edge loops.  That is, there is no edge of $\GG_\sigma$ connecting a vertex to itself.

A knot or link is called \emph{semi--adequate} if it admits a diagram for which the all--$A$ or all--$B$ state $\sigma$ is adequate.
\end{define}

Note that the link in Figure~\ref{fig:ex-state} is $A$--adequate, but not $B$--adequate, because a segement at the top of the graph of $H_B$ connects a state circle to itself.  

Semi-adequate links appeared first in the study of Jones--type invariants \cite{lick-thistle, thi:adequate}, and consist of a broad class of links.  These links were the main focus of study in \cite{fkp}.  However, that paper also studied broader classes of links as below, which are the focus of this paper. 

\begin{define}\label{def:homogeneous}
The state circles $s_\sigma$ cut the plane into regions, and segments are placed in these regions to form $H_\sigma$.  We say that a state $\sigma$ is \emph{homogeneous} if for any region in the complement of $s_\sigma$, all segments in that region in $H_\sigma$ come from $A$--resolutions, or all segments in that region come from $B$--resolutions.

A knot or link is \emph{homogeneous} if it admits a diagram and a state $\sigma$ that is homogeneous.
\end{define}

Note that the all--$A$ and all--$B$ states will always be homogeneous, since they are composed exclusively of $A$--resolutions or $B$--resolutions, respectively.  Figure~\ref{fig:ex-homog} shows a diagram of a link and a state $\sigma$ that is homogeneous, but more general than the all--$A$ or all--$B$ state.

\begin{figure}
  \includegraphics{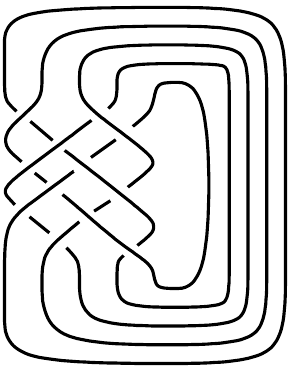}
  \hspace{.5in}
  \includegraphics{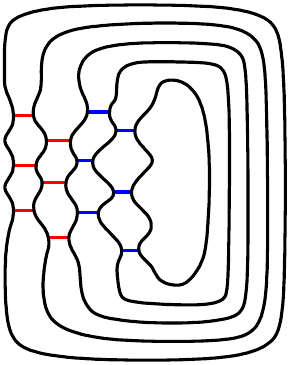}
  \caption{A link (represented as a closed 5--braid) and $H_\sigma$ for a homogeneous state $\sigma$.}
  \label{fig:ex-homog}
\end{figure}

\begin{define}\label{def:HomoAdequate}
A \emph{homogeneously adequate link} is a link that admits a diagram and a choice of state $\sigma$ that is both adequate and homogeneous.
\end{define}

In the examples above, the all--$A$ state of Figure~\ref{fig:ex-state} is both adequate and homogeneous, and the state $\sigma$ shown in Figure~\ref{fig:ex-homog} is also adequate and homogeneous.  Hence these are both homogeneously adequate links.  


In \cite{fkp}, the following theorem was obtained bounding volumes of homogeneously adequate links.

\begin{theorem}[Theorem~9.3 of \cite{fkp}]\label{thm:fkpVolumes}
Let $K$ be a link with a prime diagram that admits a homogeneously adequate state $\sigma$.  Then the volume of $S^3-K$ satisfies
\[
\vol(S^3-K) \geq v_8(\chi_-(\GG_\sigma') - ||E_c||),
\]
where $v_8 = 3.6638\dots$ is the volume of a hyperbolic regular ideal octahedron, and $\chi_-(\cdot)$ denotes the negative Euler characteristic, or $\max\{-\chi(\cdot),0\}$.  Finally, $||E_c||$ counts complex essential product disks in a spanning set.
\end{theorem}

The term $||E_c||$ will be defined carefully below.  For now, note that the term $\chi_-(\GG_\sigma)$ can easily be read off the prime, homogeneously adequate diagram of the link $K$.  It is more difficult to determine $||E_c||$.  Nevertheless, we show that $||E_c||$ is bounded by properties of the diagram of a homogeneously adequate link.  To do so, we will need to review the polyhedral decomposition in \cite{fkp}.

\subsection{Polyhedral Decomposition}

In \cite{fkp}, it was shown how to turn a homogeneously adequate diagram of a link into a collection of checkerboard ideal polyhedra.  Not all the details of the decomposition into polyhedra will be necessary here.  However, we will review properties of the polyhedra that we use later in this paper.

Given a state $\sigma$, recall that the state circles $s_\sigma$ cut the plane of projection into distinct regions.  We say a region is \emph{nontrivial} if it contains more than one state circle on its boundary.  In Figure~\ref{fig:alt-subdiagram}, an example of a graph $H_\sigma$ is shown on the left, as well as a nontrivial region, second from the left.

\begin{figure}
  \includegraphics{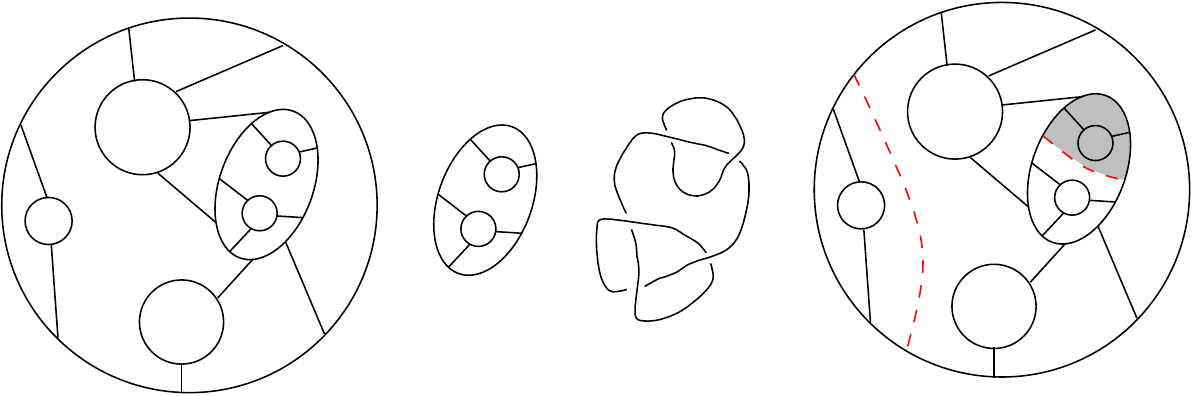}
  \caption{Left to right: A graph $H_\sigma$; a nontrivial region; the associated alternating sub-diagram (assuming the region is an $A$--region);
a maximal collection of non-prime arcs.  The shaded area indicates one polyhedral region (of four polyhedral regions total). }
  \label{fig:alt-subdiagram}
\end{figure}

Let $D(K)$ be a diagram and $\sigma$ a state such that $\sigma$ is homogeneously adequate.  For any nontrivial region in the complement of $s_\sigma$, one may form an alternating sub-diagram of $D(K)$ by starting with the state circles of $s_\sigma$ on the boundary of that region, and adding to the state circles the crossings associated with segments contained inside that region.  Figure~\ref{fig:alt-subdiagram}, third from left, shows an example.  Note that homogeneity of the diagram ensures that the result is alternating.  However, the resulting alternating subdiagram may not be prime.  

\begin{define}\label{def:prime}
Recall that a diagram is \emph{prime} if every simple closed curve meeting the diagram transversely in exactly two edges bounds a region of the diagram with no crossings.
\end{define}

Menasco described a polyhedral decomposition for prime alternating diagrams \cite{menasco:polyhedra}, which we wish to use.  If our alternating subdiagram is not prime, there is an arc in $H_\sigma$, running through the associated nontrivial region, with both endpoints on the same state circle $C$, which arc meets no other state circles, but splits the given region into two sub-regions, both containing state circles besides $C$.

\begin{define}\label{def:NonprimeArc}
An arc as above is called a \emph{non-prime arc}.  That is, it is an arc embedded in the complement of $H_\sigma$ with both endpoints on a state circle $C$, splitting a region of the complement of $s_\sigma$ into two sub-regions, both of which contain segments of $H_\sigma$.

We say a collection of non-prime arcs $\alpha_1, \dots, \alpha_n$ is \emph{maximal} if, for each nontrivial component in the complement of $s_\sigma$, the union of the arcs $\alpha_1, \dots, \alpha_n$ split the corresponding component of $H_\sigma$ into components, each of which contains segments, and moreover, the addition of any other non-prime arc yields such a component that does not contain segments. 
\end{define}

A maximal collection of non-prime arcs is shown in the above example, Figure~\ref{fig:alt-subdiagram}, right.

\begin{define}\label{def:PolyhedralRegion}
Let $\alpha_1, \dots, \alpha_n$ be a maximal collection of non-prime arcs.  A region in the complement of the union of the state circles $s_\sigma$ and the non-prime arcs $\bigcup_i \alpha_i$ is called a \emph{polyhedral region}.  
\end{define}

See Figure~\ref{fig:alt-subdiagram} for an example.  As in this example, typically we think of a polyhedral region as living in $H_\sigma$.  That is, add to the region the segments of $H_\sigma$.  Note that because $\sigma$ is a homogeneous state, all the segments will come from $A$--resolutions, or all of them will come from $B$--resolutions in the polyhedral region.

\begin{define}\label{def:LowerPolyhedron}
Given a polyhedral region, form an associated prime alternating diagram by taking circles corresponding to the boundary of the region and adding to these a crossing for each segment in the region.  Associated to any such alternating diagram is a checkerboard colored ideal polyhedron, obtained by taking the polyhedron with the diagram graph on its boundary, and vertices removed (at infinity).  This is the usual checkerboard colored ideal polyhedron associated with a prime alternating link, as described by Menasco \cite{menasco:polyhedra}.  We call it the \emph{lower polyhedron} associated with the polyhedral region.  
\end{define}

The lower polyhedra form the portions of our polyhedral decomposition that are easiest to understand, since they are standard polyhedra of alternating links.  They don't play much of a role in this paper.  There is one additional, more complicated polyhedron, the \emph{upper polyhedron}, which plays the most important role in this work. Although it is more complicated than the lower polyhedra, all the properties of the upper polyhedron can be read off of the graph $H_\sigma$.  We spend some time now describing the upper polyhedron in terms of $H_\sigma$.

The upper polyhedron consists of white faces and shaded faces.  White faces will be the components in the complement of shaded faces.  

The shaded faces are made up of three pieces, namely \emph{innermost disks}, \emph{tentacles}, and \emph{non-prime switches}.  We first define these three parts, then describe how they fit together.

\begin{define}\label{def:InnermostDisk}
Each trivial region in the complement of $s_\sigma$ is bounded by a single state circle.  We call a trivial region an \emph{innermost disk}.
\end{define}

\begin{define}\label{def:Tentacle}
A \emph{tentacle} is a portion of shaded face running adjacent to exactly one segment and state circle.  A tentacle is contained in a single polyhedral region.  If the polyhedral region is an $A$--region, meaning all segments come from $A$--resolutions, then the tentacles in that region have the shape of an $L$: The \emph{head}, or vertical portion of the $L$, lies against the segment, and the \emph{tail}, or horizontal portion of the $L$, lies against the state circle.  The tail terminates when it meets a segment lying on the same side of the state circle.  We say a tentacle in an $A$--region runs down and right, or right--down.  Figure~\ref{fig:PD1}, middle, shows right--down tentacles in an example.

If the polyhedral region is a $B$--region, then the tentacles in the region have the shape of a $\Gamma$: The \emph{head}, or vertical portion of the $\Gamma$, lies against a segment.  The \emph{tail}, or horizontal portion, runs adjacent to a state circle.  It terminates when it meets another segment on the same side.  We say a tentacle in a $B$--region runs down and left, or left--down.
\end{define}

\begin{define}\label{def:NonPrimeSwitch}
For any non-prime arc, a \emph{non-prime switch} is a portion of a shaded face forming a regular neighborhood of the non-prime arc in $H_\sigma$.
\end{define}

\begin{define}\label{def:UpperPolyhedron}
Starting with the graph of $H_\sigma$, form the \emph{upper polyhedron} as follows.
\begin{enumerate}
\item Remove bits of state circles at the top--right and bottom--left of $A$--segments and at the top--left and bottom--right of $B$--segments, as in Figure~\ref{fig:PD1}, left.
\item Add maximal collection of non-prime arcs $\alpha_1, \dots, \alpha_n$ (shown as dashed red lines in Figure~\ref{fig:PD1} left).
\item Color each innermost disk a unique color.
\item Add tentacles as follows: Tentacles exit state circles at gaps and flow right--down into $A$--regions or left--down into $B$--regions, and they are given the same color as the shaded face on the opposite side of the initial gap.  They terminate when they meet another segment. See Figure~\ref{fig:PD1}, middle.
\item Join tentacles by non-prime switches, merging two different colored faces into one, as in Figure~\ref{fig:PD1}, right.
\end{enumerate}
The above process defines the shaded faces of the upper polyhedron.  In \cite{fkp}, it was shown that the shaded faces are simply connected.  

The rest of the upper polyhedron consists of the following data.
\begin{itemize}
\item The white faces are the remaining, unshaded regions of the projection plane.  They are in one--to--one correspondence with regions in the complement of $H_\sigma \cup(\cup_i \alpha_i)$.
\item Ideal edges separate white and shaded faces.  
\item Ideal vertices are remnants of $H_\sigma$.  Note they zig-zag in $H_\sigma$, following portions of state circles and segments.  We often call them \emph{zig-zag vertices}.  
\end{itemize}
\end{define}

\begin{figure}
\begin{tabular}{ccc}
\includegraphics[width=1.7in]{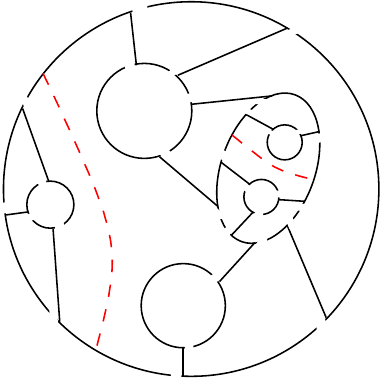} &
\includegraphics[width=1.7in]{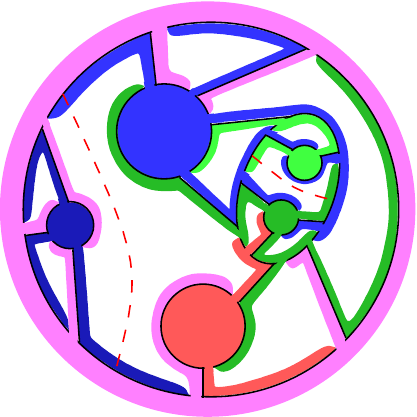} &
\includegraphics[width=1.7in]{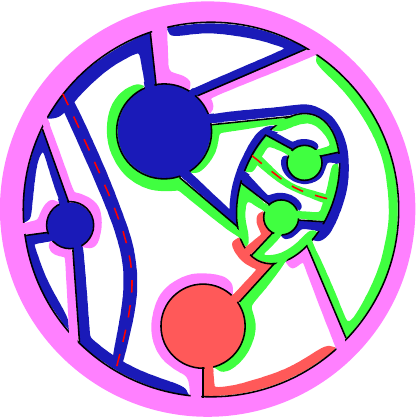}
\end{tabular}
\caption{Constructing shaded faces.
  Left: From $H_\sigma$, remove bits of state circles (solid lines) and add in non-prime arcs (dashed).  Middle: Color the innermost disks and extend tentacles.  Right: Add non-prime switches.}
\label{fig:PD1}
\end{figure}

In \cite{fkp}, it was shown that the upper and lower polyhedra give a decomposition of the manifold, denoted $M_\sigma$, obtained by cutting the knot complement along a particular incompressible surface, the \emph{state surface $S_\sigma$}.  For this paper, we do not need a description of that surface or the cut manifold $M_\sigma$, only the results that follow from the polyhedral decomposition, and so we omit the details here. 

A shaded face of the upper polyhedron determines a directed graph.  Innermost disks are sources, which correspond to vertices of the directed graph.  Tentacles can be represented by directed edges running from head to tail.  A non-prime switch is given by four directed edges, running in the direction of adjacent tentacles, and a vertex.  
Starting in Section~\ref{sec:tentacles}, we will frequently consider directed arcs inside of shaded faces, and think of them as running through the directed graph.  We need a couple of definitions to explain their direction.  We also want to ensure the arcs we consider are simplified, in the sense that they don't double back on themselves. These are the reasons for the following definitions.

\begin{define}\label{def:Downstream}
We say that an arc running monotonically through the directed graph of a shaded face is \emph{simple with respect to the shaded face}.  Any arc in a shaded face may be isotoped to be simple.  

A simple arc running through a tentacle is said to run \emph{downstream} if it runs in the direction of the tentacle, i.e.\ from head to tail.  Otherwise, it runs \emph{upstream}.  A simple arc running through a shaded face can only change direction, downstream to upstream, inside a non-prime switch.
\end{define}

We are now ready to define the complex essential product disks of Theorem~\ref{thm:fkpVolumes}.  We only need to describe the notion of an essential product disk in the upper polyhedron.  

\begin{define}\label{def:EPDinUpper}
In the upper polyhedron, an \emph{essential product disk (EPD)} is a properly embedded essential disk whose boundary meets distinct ideal vertices of the upper polyhedron (remnants of $H_\sigma$) exactly twice, and which lies completely in two colored faces otherwise.
\end{define}

\begin{define}\label{def:ComplexEPD}
An EPD in the upper polyhedron is said to be \emph{simple} if its boundary encloses a single white bigon face of the polyhedron.  It is said to be \emph{semi-simple} if its boundary encloses a union of white bigon faces to one side, and no other white faces to that side.  It is \emph{complex} otherwise.
\end{define}

\begin{define}\label{def:EPDSpan}
A finite collection of disjoint EPDs $E_1, \dots, E_n$ \emph{spans} a set $B$ if $B\setminus(E_1 \cup \dots \cup E_n)$ is a finite collection of prisms ($I$--bundles over a polygon) and solid tori. 
\end{define}

It was shown in \cite{fkp} that a characteristic $I$--bundle associated to the knot complement is spanned by EPDs completely contained in the ideal polyhedra of the polyhedral decomposition.  Moreover, there exists a set $E_s \cup E_c$ of EPDs in the upper polyhedron $P_\sigma$ that spans the intersection of the $I$--bundle with $P_\sigma$, \cite[Lemma~5.8]{fkp}.  Here $E_s$ consists of all simple disks in the upper polyhedron, and $E_c$ consists of complex disks.  The collection $E_c$ is minimal in the following sense.  

Suppose the boundary of an EPD runs adjacent to both sides of an ideal vertex without meeting that vertex. Then, by connecting the boundary across this vertex, and surgering the EPD, we obtain two new EPDs.  We say the two EPDs are obtained by a \emph{parabolic compression} of the original.  The collection $E_c$ is minimal in the sense that no EPD in $E_c$ parabolically compresses to a subcollection of $E_s \cup E_c$.  

\begin{define}\label{def:Ec}
The term $||E_c||$ in Theorem~\ref{thm:fkpVolumes} is the number of complex EPDs in the collection $E_c$.
\end{define}

We wish to bound $||E_c||$ in terms of the diagram.

\section{Useful Lemmas for Tentacle Chasing}\label{sec:tentacles}

Our goal now is to find EPDs in the upper polyhedron.  In order to do so, we will use combinatorics of the graph $H_\sigma$ and combinatorics of the upper polyhedron, which we visualize as lying on the graph $H_\sigma$.  Many of our tools come from \cite{fkp}.  In this section, we review the main tools we will be using.

The first set of results come from \cite[Chapter~3]{fkp}.  In that paper, these lemmas are called Escher Stairs, Shortcut Lemma, Downstream Lemma, and the Utility Lemma.  The proofs in that paper were originally written in the semi-adequate case rather than the homogeneously adequate case.  However, it was observed there that the proofs carried over into the homogeneously adequate case (see section 3.4 of that paper).  We state the results here for reference, and include brief proofs as a warm up to the combinatorial proof technique of tentacle chasing.  First, we need a definition.

\begin{define}\label{def:Staircase}
A \emph{staircase} is a connected subgraph of $H_{\sigma}$ consisting of finitely many segments alternating with portions of state circles, which has the following properties.  First, two segments with a state circle between them lie on opposite sides of the state circle.  Second, segments and state circles in the staircase are given directions, such that the graph runs monotonically from ``top'' to ``bottom.''  A \emph{step} consists of a segment and the portion of state circle in the direction of the segment.  If the segment lies in an $A$--region, then the state circle runs right from the end of the step.  If it lies in a $B$--region, it runs left.  See Figure~\ref{fig:Staircase}.
\end{define}

\begin{figure}
  \includegraphics{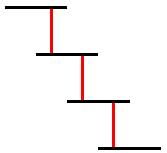} \hspace{.3in}
  \includegraphics{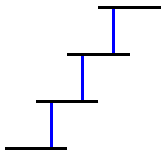} \hspace{.3in}
  \includegraphics{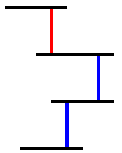}
  \caption{Left to right: A staircase with all segments in $A$--regions.  A staircase with all segments in $B$--regions.  A mixed staircase.}
  \label{fig:Staircase}
\end{figure}

\begin{lemma}[Escher stairs]\label{lemma:escher}
Let $K$ be a link with a diagram and a state $\sigma$ that is homogeneous and adequate.  In the graph $H_{\sigma}$, the following hold.
\begin{enumerate}
\item There is no staircase in which one state circle $C$ is both the top and the bottom of the staircase, with top and bottom steps on the same side of $C$.  
\item No staircase forms a loop, i.e.\ top and bottom steps lie on opposite sides of the same state circle.
\end{enumerate}
\end{lemma}

See Figure \ref{fig:Escher}.

\begin{figure}
  \includegraphics{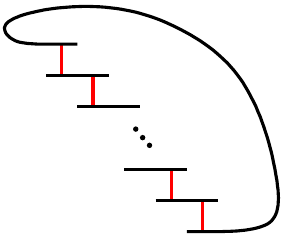} \hspace{.3in}
  \includegraphics{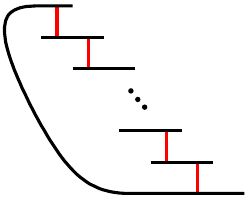}
\caption{Left: a right--down staircase forming a loop. Right: a right--down staircase with top and bottom on same side of the same state circle.}
\label{fig:Escher}
\end{figure}

\begin{proof}
Suppose by way of contradiction that such a staircase exists.  Notice that the segments and state circles of the staircase form a closed curve $\gamma$ in the diagram.  The state circles of the staircase intersect $\gamma$.  Since the state circles are also closed curves, they must connect to other state circles of the staircase on both sides of $\gamma$.  We show by induction on the number of stairs of the staircase that this situation inevitably violates the adequacy of the diagram.

Consider the first result: the staircase in which one state circle $C$ is both the top and bottom, with top and bottom segments on the same side of $C$.  Let $C_1, \dots, C_n$ denote the other state circles in the staircase.  If $n$ is odd, then not all state circles of the staircase can connect to each other, violating our observation above.  So $n$ must be even.  If $n=2$, $C_1$ and $C_2$ must be connected.  However, there is a segment of the staircase between $C_1$ and $C_2$, contradicting the adequacy of this state.  See Figure~\ref{fig:EscherS2}.  Therefore, there cannot exist a staircase with only three state circles in which one of the state circles is the top and bottom of the staircase.

\begin{figure}
  \includegraphics{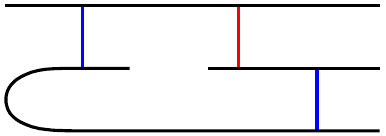} \hspace{.2in}
  \includegraphics{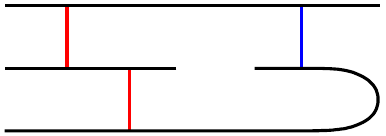}
\caption{Examples of staircases with three state circles, and top and bottom on the same state circle.}
\label{fig:EscherS2}
\end{figure}

We now proceed with the inductive step.  Assume by induction that it is impossible to form a staircase with top and bottom on the same step with $p$ state circles, for $p=2,3, \dots,n$.  We show that for $p=n+1$ the scenario is also impossible. Let the state circles be denoted $C,C_1,C_2,\dots,C_{n-1},C_n$, with $C$ at the top and bottom.  As before, we must connect the loose ends of the state circles $C_i$ enclosed by $\gamma$.  The state circle $C_n$ must connect to some state circle $C_j$ inside $\gamma$.  This connection cuts the staircase into sub-staircases, one of which is a staircase in which the state circle $C_n = C_j$ is the top and bottom of the staircase.  See Figure~\ref{fig:EscherI}.  Now the number of state circles in this sub-staircase is at most $n-2$.  This is impossible.

\begin{figure}
  \includegraphics{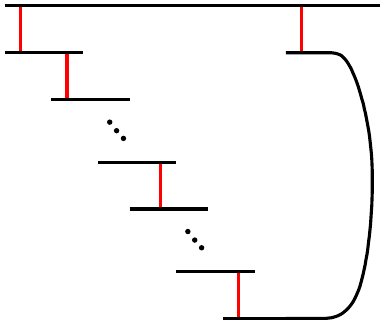} \hspace{.3in}
  \includegraphics{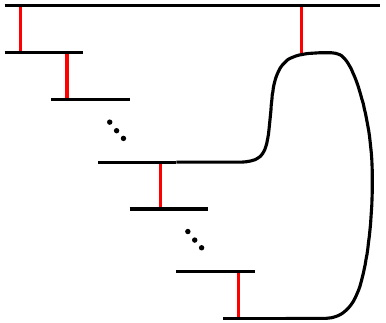}
\caption{Left: a staircase with $n+1$ state circles with $C$ at top and bottom.  Right: there exists a sub-staircase with $C_n$ at the top and bottom.}
\label{fig:EscherI}
\end{figure}

We now prove the second result of the lemma, that no staircase may form a loop.  Suppose that $C_1, \dots, C_n$ are state circles in a staircase forming a loop, with top and bottom segments on opposite sides of $C_1$.  The staircase forms a simple closed curve $\gamma$ in the plane.  Hence state circles $C_1, \dots, C_n$ must connect to each other on either side of $\gamma$.  Then $C_n$ must connect to some state circle $C_j$.  But now the sub-staircase from $C_j$ to $C_n$ has top and bottom segments on the same state circle $C_n=C_j$, on the same side.  This contradicts our previous result. 
\end{proof}

Recall that a non-prime arc has endpoints on a state circle $C$, and divides a region bounded by $C$ in the plane into two regions.  We say that each such region is a \emph{non-prime half--disk}.

\begin{lemma}[Shortcut lemma]\label{lemma:shortcut}
Let $\alpha$ be a non-prime arc with endpoints on a state circle $C$ in a diagram with homogeneous, adequate state $\sigma$.  Suppose a directed arc $\gamma$ lies entirely inside a single shaded face, and is simple with respect to that shaded face.  Suppose $\gamma$ runs across $\alpha$ into the interior of the non-prime half--disk bounded by $\alpha$ and $C$, and then runs upstream.  Finally, suppose that $\gamma$ exits the interior of that half--disk across the state circle $C$.  Then $\gamma$ must exit by following a tentacle downstream (that is, it cannot exit running upstream).
\end{lemma}

\begin{proof}
Consider an innermost counterexample.  That is, assume that $\gamma$ exits the non-prime half--disk bounded by non-prime arc $\alpha$ by running upstream, but if $\gamma$ enters into any other non-prime half--disk after crossing $\alpha$ and before exiting the half--disk bounded by $\alpha$, then it exits running downstream.  According to the hypothesis, upon crossing non-prime arc $\alpha$, $\gamma$ turns upstream through a tentacle.  Note that this tentacle may be in either an $A$ or $B$ region. Once $\gamma$ has begun running upstream, $\gamma$ may continue travelling in the following ways: (a) upstream, or (b) downstream, including the case that it runs into a non-prime half--disk.

For the first case, $\gamma$ continues upstream until it meets another state circle and encounters the same options of continuation. By Lemma~\ref{lemma:escher}, $\gamma$ may not exit $\alpha$ by continuing to run upstream.  Thus $\gamma$ must, at some point, run downstream.

Consider the second case. We claim that once $\gamma$ begins running downstream it must continue downstream until it exits the state circle $C$. If $\gamma$ crosses a state circle going downstream, it will still be running downstream in the new tentacle, because it is simple.  If $\gamma$ crosses into another non-prime half--disk it must eventually exit running downstream, even if it runs upstream at some point within the half--disk, by the assumption that $\alpha$ is the innermost counterexample.  The arc $\gamma$ may also run over a non-prime arc without entering the half--disk it bounds, but this results in no change of direction.

Therefore, in all cases, $\gamma$ must exit $\alpha$ by crossing $C$ by running downstream.
\end{proof}

Lemma~\ref{lemma:shortcut} implies the following, stronger lemma, whose proof we omit, but is contained in \cite{fkp}.

\begin{lemma}[Downstream]\label{lemma:downstream}
Let $\gamma$ be a directed arc lying in a single shaded face that is simple with respect to the shaded face.  Suppose $\gamma$ begins by crossing a state circle running downstream, and suppose that every time $\gamma$ crosses a non-prime arc, it exits the corresponding non-prime half--disk.  Then $\gamma$ defines a staircase such that every segment of the staircase is adjacent to $\gamma$, with $\gamma$ running downstream.  Moreover, the endpoints of $\gamma$ lie on tentacles adjacent to the first and last stairs.  
\end{lemma}

Finally, the following lemma is most useful for tentacle chasing arguments.

\begin{lemma}[Utility lemma]\label{lemma:utility}
Let $\gamma$ be a simple directed arc in a shaded face that starts and ends on the same state circle $C$. Then the following are true:
\begin{enumerate}
\item\label{item:utility1} $\gamma$ starts by running $upstream$ away from $C$,
\item\label{item:utility2} $\gamma$ terminates running $downstream$ toward $C$, and
\item\label{item:utility3} $\gamma$ cannot intersect $C$ more than two times.
\end{enumerate}
\end{lemma}

\begin{proof}
We will consider three results individually and prove each by contradiction.

For item~\eqref{item:utility1}, assume that $\gamma$ begins by running downstream, contrary to the claim. By Lemma~\ref{lemma:downstream}, there exists a staircase starting on $C$ such that $\gamma$ runs downstream, adjacent to each segment of the staircase.  Since $\gamma$ starts and ends on $C$, the staircase must start and end on $C$.  This contradicts Lemma~\ref{lemma:escher} (Escher Stairs).

For item~\eqref{item:utility2}, we know $\gamma$ begins at $C$ running upstream.  Assume that $\gamma$ terminates at $C$ running upstream.  By reversing the orientation of $\gamma$, we obtain the same scenario as in part \eqref{item:utility1}.  Therefore, $\gamma$ must terminate running downstream.

For item~\eqref{item:utility3}, assume $\gamma$ crosses $C$ more than once, at points $x_1$, $x_2$, $\dots$, $x_n$. Consider the sub arc of $\gamma$ from $x_1$ to $x_2$. By item~\eqref{item:utility2}, the subarc must cross $C$ going downstream at $x_2$. However, this means that for the subarc of $\gamma$ from $x_2$ to $x_3$, the arc travels downstream from $C$ at $x_2$, which contradicts item~\eqref{item:utility1}. Therefore, $\gamma$ may only cross $C$ two times.
\end{proof}

\section{Essential Product Disks}\label{sec:EPDs}

The following theorem, whose proof is contained in this section, is the main result of this paper.  It generalizes \cite[Theorem~6.4]{fkp}, which was proved only for $A$--adequate links.  Our proof is similar to the proof in \cite{fkp}, following the same sequence of steps.  However we need to include more cases when we allow both $A$ and $B$ resolutions in homogeneously adequate diagrams.  

\begin{theorem}\label{thm:main}
Let D(K) be a prime, $\sigma$--adequate $\sigma$--homogeneous diagram of a link $K$ in $S^3$.  Denote by $P_\sigma$ the upper polyhedron in the prime polyhedral decomposition described above.  Suppose there is an essential product disk $E$ embedded in $P_\sigma$. Then there is a 2--edge loop in $\mathbb{G}_{\sigma}$ so that $\partial E$ runs over tentacles adjacent to segments of the 2--edge loop.
\end{theorem}

The $A$--adequate versions of many of the following lemmas appear in \cite[Chapter~6]{fkp}.  We begin with a lemma which allows us to work with normal squares in the upper polyhedron, rather than EPDs.  This is the analogue to \cite[Lemma~6.1]{fkp}, and its proof is identical.

\begin{lemma}[EPD to oriented square]\label{lemma:EPDtoSquare}
Let $D(K)$ be a prime, $\sigma$--adequate and $\sigma$--homogeneous diagram of a link in $S^3$.  Let $P_\sigma$ denote the upper polyhedron.  Suppose there is an EPD $E$ properly embedded in $P_\sigma$.  Then $\partial E$ can be pulled off the ideal vertices of $P_\sigma$ to give a normal square $E'$ in $P_\sigma$ with the following properties.
\begin{enumerate}
\item Two opposite edges of $\partial E'$ run through shaded faces, which we label green and orange.
\item The other two opposite edges run through white faces, each cutting off a single vertex of the white face.
\item The single vertex cut off in a white face forms a triangle in that white face, oriented in such a way that in counter-clockwise order, the edges of the triangle are colored orange--green--white.
\end{enumerate}
Finally, the two white edges of the normal square cannot lie on the same white face of the polyhedron.
\end{lemma}

\begin{proof}
The EPD $E$ runs over two shaded faces, labeled green and orange, and runs between them over two ideal vertices.  If we pull $\partial E$ off of the ideal vertex slightly, into one of the two white faces, the result will cut off a triangle.  The orientation is determined by which side of the vertex $\partial E$ is pulled into, and we choose the orientation that matches item (3) of the lemma.  The white faces must be distinct else the boundary of the normal square $E'$ on one of the shaded faces can be joined to an arc between the two white sides of $\partial E'$ to form a normal bigon, contradicting the fact that $P_\sigma$ admits no normal bigons \cite[Proposition~3.24]{fkp}.  
\end{proof}

In the above lemma, the term \emph{normal} is used without definition.  A precise definition is given in \cite[Definition~3.14]{fkp}.  However, we will only work with normal squares that have the form of Lemma~\ref{lemma:EPDtoSquare}.  

\begin{remark}\label{rmk:OrientedSquare}
When a normal square has the form of Lemma~\ref{lemma:EPDtoSquare}, note that the white faces will lie at the tails of orange tentacles and heads of green ones in the $A$--regions, but they will lie at the tails of green and heads of orange in the $B$--regions.  Examples of portions of these squares are shown in the figures in this section.  See, for example, Figure~\ref{fig:lollipops}.
\end{remark}

To prove Theorem~\ref{thm:main}, we start with an EPD and pull it into a normal square following Lemma~\ref{lemma:EPDtoSquare}, and then analyze such normal squares.  During the course of the proof, we will see that the first 2--edge loop near a white face of the square has one of seven possible forms, shown in Figures~\ref{fig:EPDformA} through \ref{fig:EPDformG}.  These are analogous to Types $\mathcal{A}$ through $\mathcal{G}$ in \cite[Figure~6.1]{fkp}, although more complicated behavior occurs here.  In that paper, all regions were $A$--regions; in this one, regions may switch from $A$--regions to $B$--regions and vice versa.  
We record the results in a proposition.

\begin{prop}\label{prop:EPDTypes}
Suppose $E$ is an EPD in a prime diagram of a link $K$ with an adequate, homogeneous state $\sigma$.  Then the corresponding normal square $E'$ of Lemma~\ref{lemma:EPDtoSquare} has boundary running over tentacles adjacent to a 2--edge loop.  Moreover, it has a white side whose nearest 2--edge loop has one of the forms in Figures~\ref{fig:EPDformA} through \ref{fig:EPDformG}, where in each general case, the dots indicate a staircase (with possibly no segments) of one of the forms of Lemma~\ref{lemma:FullStaircase}, shown in Figure~\ref{fig:staircases}.
\end{prop}

\begin{figure}
  \begin{tabular}{cccc}
  Elementary & & General & \\
  \includegraphics{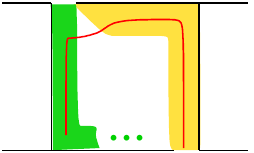} &
  \includegraphics{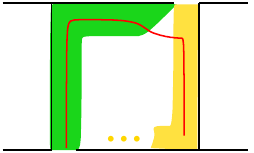} &
  \includegraphics{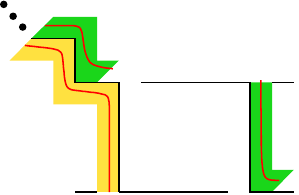} &
  \includegraphics{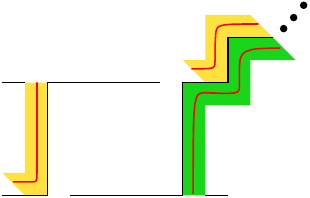} \\
  $A$--region &
  $B$--region &
  $A$ to $A$--region &
  $B$ to $B$--region \\
  \end{tabular}
  \caption{Type $\mathcal{A}$:  EPD runs over tentacles of different colors adjacent to 2--edge loop, white side cuts off a triangle with vertex (zig-zag) containing one of the segments of the loop.}
  \label{fig:EPDformA}
\end{figure}

\begin{figure}
  \begin{tabular}{cc}
    Elementary & General \\
    \begin{tabular}{cc}
      \includegraphics{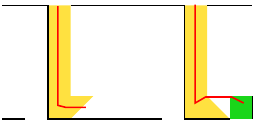} &
      \includegraphics{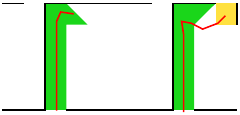} \\
      \includegraphics{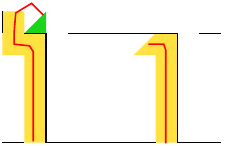} &
      \includegraphics{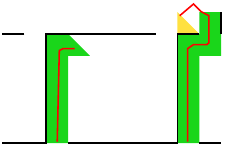} \\
      $A$--region & $B$--region
    \end{tabular}
    &
    \begin{tabular}{cc}
      \includegraphics{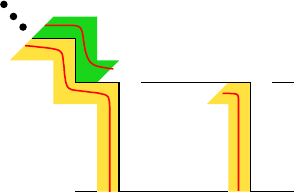} &
      \includegraphics{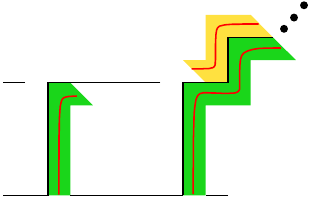} \\
      $A$ to $A$--region & $B$ to $B$--region
    \end{tabular}
  \end{tabular}
  \caption{Type $\mathcal{B}$: EPD runs over tentacles of same color on 2--edge loop, white side cuts of triangle with vertex (zig-zag) containing one segment in the loop.}
  \label{fig:EPDformB}
\end{figure}

\begin{figure}
  \begin{tabular}{cccc}
    Elementary & & General & \\
    \includegraphics{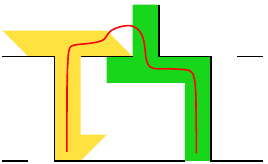} &
    \includegraphics{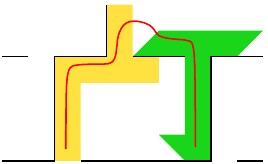} \hspace{.5in} &
    \includegraphics{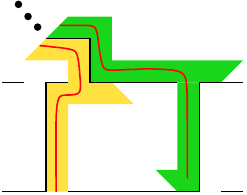} &
    \includegraphics{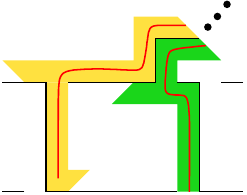} \\
    $A$--region & $B$--region & $A$ to $B$--region & $B$ to $A$--region \\
  \end{tabular}
  \caption{Type $\mathcal{C}$:  Tentacles of different color, white side cuts off vertex not containing a segment of 2--edge loop, not separated by non-prime arc.}
  \label{fig:EPDformC}
\end{figure}

\begin{figure}
  \begin{tabular}{cccc}
    Elementary & & General & \\
    \includegraphics{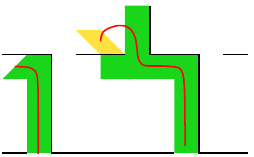} &
    \includegraphics{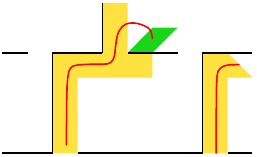} \hspace{.5in} &
    \includegraphics{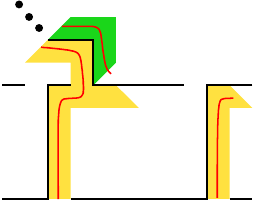} &
    \includegraphics{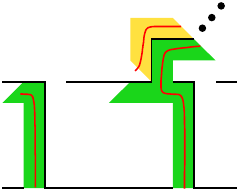} \\
    $A$--region & $B$--region & $A$ to $B$--region & $B$ to $A$--region 
  \end{tabular}
  \caption{Type $\mathcal{D}$:  Tentacles of same color, white side cuts off vertex not containing segment of 2--edge loop, and not separated from 2--edge loop by non-prime arc.}
  \label{fig:EPDformD}
\end{figure}

\begin{figure}
  \begin{tabular}{cc}
    \includegraphics{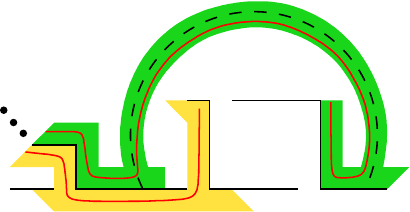} &
    \includegraphics{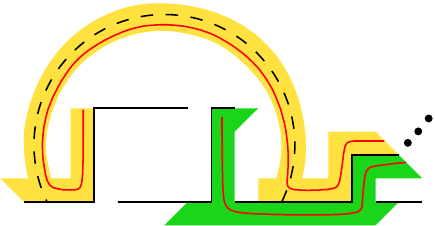} \\
    $A$--region & $B$--region
  \end{tabular}
  \caption{Type $\mathcal{E}$: Tentacles opposite color, white side separated from 2--edge loop by non-prime arc.}
  \label{fig:EPDformE}
\end{figure}

\begin{figure}
  \begin{tabular}{cc}
    \includegraphics{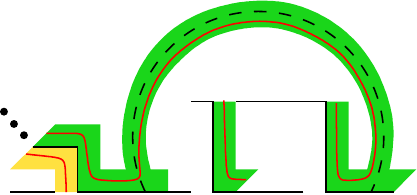} &
    \includegraphics{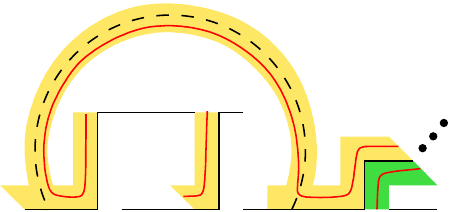} \\
    $A$--region & $B$--region
  \end{tabular}
  \caption{Type $\mathcal{F}$: Tentacles of same color, white side separated from 2--edge loop by non-prime arc on same side of state circle as vertex (zig-zag).}
  \label{fig:EPDformF}
\end{figure}

\begin{figure}
  \begin{tabular}{cccc}
    \includegraphics{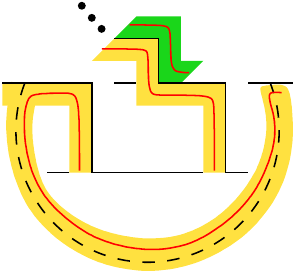} &
    \includegraphics{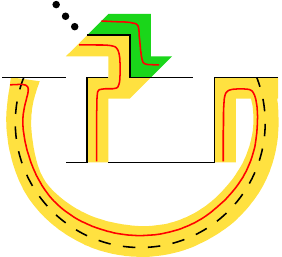} &
    \includegraphics{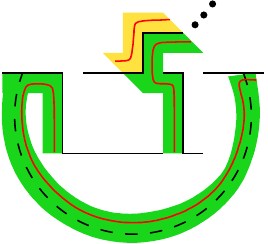} &
    \includegraphics{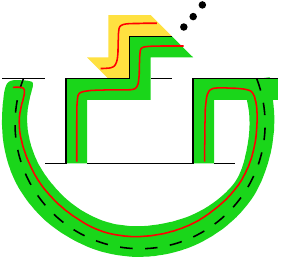} \\
    $A$ to $A$--region & $A$ to $B$--region & $B$ to $A$--region & $B$ to $B$--region
  \end{tabular}
  \caption{Type $\mathcal{G}$:  Tentacles of same color, white side separated from 2--edge loop by non-prime arc, non-prime arc on opposite side of state circle as vertex (zig-zag).}
  \label{fig:EPDformG}
\end{figure}

First, we deal with the case that the two white faces of the normal square both lie in the same polyhedral region.  This case is simpler, because we may map into the corresponding lower polyhedron using a map called the ``clockwise'' or ``counter-clockwise'' map, and then apply standard results on alternating diagrams.  We obtain the following.

\begin{lemma}[White faces in same region]\label{lemma:SameRegion}
Suppose $D(K)$ is a prime diagram with an adequate, homogeneous state $\sigma$ and corresponding upper polyhedron $P_\sigma$.  Suppose $E$ is an EPD in $P_\sigma$, and $E'$ is the corresponding normal square from Lemma~\ref{lemma:EPDtoSquare}, intersecting white faces $V$ and $W$ in arcs $\beta_V$ and $\beta_W$, respectively.  If $V$ and $W$ are in the same polyhedral region, then $\partial E'$ runs over tentacles of a 2--edge loop in the upper polyhedron of type $\mathcal{B}$, Figure~\ref{fig:EPDformB}.
\end{lemma}

The proof is nearly identical to that of \cite[Lemma~6.7]{fkp}.  We step through it briefly to pick up the form of the 2--edge loop.

\begin{proof}
Observe that since $V$ and $W$ are in the same polyhedral region and our diagram is $\sigma$--homogeneous, then the region will either have all $A$ or all $B$ resolutions. 

Apply the clockwise map of \cite[Lemma~4.8]{fkp} for an $A$--region, or the counter-clockwise map for a $B$--region (see also \cite[Lemma~3.2]{fkp:qsf} for both).  By the cited results, we obtain a square $S'$ in the lower polyhedron cutting off a single vertex in each white side.  This square is either inessential, meaning it bounds a single ideal vertex in the lower polyhedron, or it can be isotoped to an EPD in the lower polyhedron.
In the case that the square is essential, \cite[Lemma 5.1]{fkp} implies that $S'$ runs over two segments of $H_\sigma$ corresponding to a 2--edge loop.  In the case that $S'$ is inessential, its white arcs must cut off the same ideal vertex with opposite orientation, which implies one of the arcs lies on a bigon face.  In either case, $S'$ runs over segments of $H_\sigma$ corresponding to a 2--edge loop.  (See, e.g.\ \cite[Figure~6.3]{fkp}.)

When the region is an $A$ region, the proof of \cite[Lemma~6.7]{fkp} shows that $E'$, the original square, runs adjacent to a 2--edge loop as in the left of Figure~6.3 of that paper.  When the region is a $B$ region, the argument still applies, although the figure must be reflected and colors interchanged.  In either case, $\partial E'$ runs adjacent to two tentacles of the same color (orange for $A$, green for $B$), running upstream.  This is type $\mathcal{B}$, elementary, shown on the top left of Figure~\ref{fig:EPDformB}.
\end{proof}

The next two lemmas help us identify a 2--edge loop when the EPD runs adjacent to distinct points on a state circle.  We will apply them several times.

The first is identical to \cite[Lemma~6.8]{fkp}, both in its statement and its proof, using the fact that the Utility Lemma, Lemma~\ref{lemma:utility}, and the Shortcut Lemma, Lemma~\ref{lemma:shortcut}, both hold in the homogeneously adequate setting.

\begin{lemma}[Adjacent loop]\label{lemma:AdjacentLoop}
Let $\zeta$ be a directed simple arc contained in a single shaded face, adjacent to a state circle $C$ at a point $p$ on $C$. Suppose $\zeta$ runs upstream across a state circle $C'$ after leaving $p$, but then eventually continues on to be adjacent to $C$ again at a new point $p'$. Then $\zeta$ must run adjacent to two distinct segments of $H_{\sigma}$ connecting $C$ to $C'$.
\qed
\end{lemma}

The next lemma is identical to \cite[Lemma~6.9]{fkp}, also in its statement and proof.  We include a proof here that expands a bit on the first sentence of the proof there, for additional clarity.

\begin{lemma}[Adjacent points]\label{lemma:AdjacentPoints}
Suppose there are arcs $\zeta_1$ and $\zeta_2$ in distinct shaded faces in the upper polyhedron such that each $\zeta_j$ runs adjacent to two points $p_1$ and $p_2$ on the same state circle $C$ (i.e.\ $\zeta_1$ runs through a neighborhood of $p_1$ and a neighborhood of $p_2$, and similarly for $\zeta_2$).  Then either
\begin{enumerate}
\item at least one of $\zeta_1$ or $\zeta_2$ runs upstream across some other state circle and Lemma~\ref{lemma:AdjacentLoop} applies, or
\item both arcs remain adjacent to the same portion of $C$ between $p_1$ and $p_2$.
\end{enumerate}
\end{lemma}

\begin{proof}
If one of $\zeta_1$, $\zeta_2$ crosses another state circle $C'$ between points $p_1$ and $p_2$ then it must do so running upstream, by the Utility Lemma, Lemma~\ref{lemma:utility}.  That is, if $\zeta_2$ crosses $C'$ after leaving $p_1$, then it must cross $C'$ again to reach $p_2$.  The Utility Lemma implies the first crossing is in the upstream direction.  Then we are in case (1), and Lemma~\ref{lemma:AdjacentLoop} applies.

So suppose neither $\zeta_1$ nor $\zeta_2$ crosses another state circle between points $p_1$ and $p_2$.  Then we form a simple closed curve meeting $H_\sigma$ exactly twice at $p_1$ and $p_2$ by connecting portions of $\zeta_1$ and $\zeta_2$ between them.  Replacing segments of $H_\sigma$ with crossings, this gives a simple closed curve in the diagram of the link meeting the link transversely in exactly two points.  Since the diagram is prime, as in Definition~\ref{def:prime}, the curve bounds no crossings on one side.  It follows that $\zeta_1$ and $\zeta_2$ must run adjacent to $C$ on this side.
\end{proof}

We now return to proving Proposition~\ref{prop:EPDTypes}.  When we push an EPD into a normal square as in Lemma~\ref{lemma:EPDtoSquare}, the boundary of the square has two white sides, $\beta_V$ and $\beta_W$, and two shaded sides.  By Lemma~\ref{lemma:SameRegion}, the proposition will be true if $\beta_V$ and $\beta_W$ lie in the same polyhedral region.  We need to show the proposition is true if the white sides are in different polyhedral regions.  In that case, the following lemma, the analogue of \cite[Proposition~6.10]{fkp}, shows that $\beta_V$ and $\beta_W$ have a particular form.

\begin{lemma}[Hanging and standing regions]\label{lemma:lollipops}
With the hypotheses of Proposition~\ref{prop:EPDTypes}, either
\begin{enumerate}
\item the EPD runs adjacent to segments of a 2--edge loop of one of the desired types, or
\item the polyhedral regions containing $\beta_V$ and $\beta_W$ are of one of four forms, shown in Figure~\ref{fig:lollipops}.  The two forms on the left of that figure are refered to as \emph{hanging regions}, the two on the right as \emph{standing regions}, and we use $A$ or $B$ to denote whether the region is an $A$ or $B$ region.
\end{enumerate}
\end{lemma}

\begin{figure}
  \begin{tabular}{cccc}
    \includegraphics{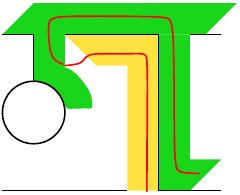} &
    \includegraphics{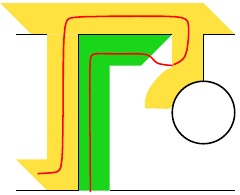} &
    \includegraphics{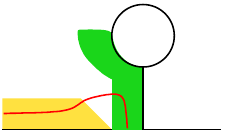} &
    \includegraphics{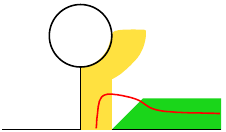} \\
    Hanging--$A$ & Hanging--$B$ & Standing--$A$ & Standing--$B$
  \end{tabular}
\caption{Possibilities for Lemma~\ref{lemma:lollipops}.  Two are in $A$--regions, two in $B$--regions.  We refer to the forms as \emph{hanging regions} and \emph{standing regions}.}
\label{fig:lollipops}
\end{figure}

The proof of Lemma~\ref{lemma:lollipops} follows that of \cite[Proposition~6.10]{fkp}, which is easily generalized to the homogeneously adequate case.  That proof is divided into four cases, but in all cases, everything of interest happens within the same polyhedral region.  Because our diagram is homogeneous, all resolutions in the region must be either all--$A$ or all--$B$.  If all--$A$, the proof of \cite[Proposition~6.10]{fkp} goes through without change.  If all--$B$, the result holds by reflection.  We walk through the argument briefly in order to determine the forms of the squares that arise.

\begin{proof}
First, by Lemma~\ref{lemma:SameRegion}, we may assume white sides of $\partial E'$, $\beta_V$ and $\beta_W$, are in distinct polyhedral regions.  Then apply \cite[Lemma~4.16]{fkp}.  This lemma is stated in the all--$A$ case in \cite{fkp}, but it concerns one fixed polyhedral region.  Thus it holds without change if our region is an $A$--region, and holds by reflection if our region is a $B$--region (this is noted in \cite[Section~4.5]{fkp}).  The lemma states that when we direct arcs in the shaded faces toward $\beta_W$, one of them, call it $\zeta_1$, runs downstream across a state circle $C_W$, and then connects to $\beta_W$ without crossing any additional state circles or non-prime arcs.  The other, call it $\zeta_2$, may either connect directly to $\beta_W$ when it enters the region of $\beta_W$, or it may run upstream first, across a different state circle $C_2$, then back downstream later across $C_2$ to connect to $\beta_W$.

Two of the cases now concern $\zeta_2$, whether it connects directly to $\beta_W$ or runs upstream first.  In each case, $\beta_W$ meets the tentacles containing $\zeta_1$ and $\zeta_2$.  We need to consider whether the head of the tentacle of $\zeta_1$ meets the tail of that of $\zeta_2$, or whether the tail of the tentacle of $\zeta_1$ meets the head of that of $\zeta_2$.  Thus there are four cases total.  We step through the cases briefly.  Refer to \cite{fkp} for complete details. 

\underline{Case 1a}: The arc $\zeta_2$ runs directly to $\beta_W$ and the head of the tentacle of $\zeta_2$ meets the tail of the tentacle containing $\zeta_1$.  Then $\zeta_1$ and $\zeta_2$ both run adjacent to segments connecting some $C'$ to $C_W$, and these segments must be distinct by the fact that the diagram is prime.  We pick up a 2--edge loop in this case that has the form of type $\mathcal{A}$, elementary, Figure~\ref{fig:EPDformA}.

\underline{Case1b}: The arc $\zeta_2$ runs directly to $\beta_W$ and the tail of the tentacle of $\zeta_2$ meets the head of the tentacle of $\zeta_1$.  This leads to one of the two standing pictures, depending on whether the region is an $A$--region or a $B$--region.

\underline{Case2a}:  The arc $\zeta_2$ runs upstream before running to $\beta_W$, and the head of the tentacle of $\zeta_1$ meets the tail of the tentacle of $\zeta_2$.  In this case, $\zeta_2$ runs upstream adjacent to a segment joining $C_W$ and some state circle $C_2$.  When $\zeta_2$ runs downstream again across $C_2$, it must run adjacent to another segment joining $C_2$ and $C_W$.  These two segments form the desired 2--edge loop.  It has the form of type $\mathcal{B}$, elementary, in Figure~\ref{fig:EPDformB}.

\underline{Case2b}:  The arc $\zeta_2$ runs upstream before running to $\beta_W$, and the tail of the tentacle of $\zeta_1$ meets the head of the tentacle of $\zeta_2$.  Again when $\zeta_2$ runs upstream, it must do so adjacent to a segment $s_2$ connecting $C_W$ and some $C_2$.  After $\zeta_2$ runs downstream across $C_2$, it connects immediately to $\beta_W$, implying that the tail of the tentacle of $\zeta_1$ lies on $C_2$.  Then $\zeta_1$ must run downstream adjacent to a segment $s_1$ connecting $C_W$ and $C_2$.  If $s_1$ and $s_2$ are distinct, we have a 2--edge loop of the form of type~$\mathcal{A}$.

If $s_1$ and $s_2$ agree, then notice that $\zeta_1$ and $\zeta_2$ both run adjacent to a point $p$ on the state circle $C_2$ where $s_1=s_2$ meets that state circle.  They also both meet at $\beta_W$, hence by shrinking $\beta_W$, we may assume they run adjacent to the point $p'$ on $C_2$ that is cut off by $\beta_W$.  Apply Lemmas~\ref{lemma:AdjacentLoop} (Adjacent Loop) and \ref{lemma:AdjacentPoints} (Adjacent points).  If $\zeta_2$ runs upstream from $C_W$, then Lemma~\ref{lemma:AdjacentLoop} implies there is a 2--edge loop.  Note it runs over tentacles of the same color.  If $\zeta_2$ runs from $\beta_W$ in an $A$--region across $C_2$ to a $B$--region, or from $\beta_W$ in a $B$--region across $C_2$ to an $A$--region, then the 2--edge loop will be of type~$\mathcal{B}$, elementary, as on the 2nd row of Figure~\ref{fig:EPDformB}.  If it runs from $\beta_W$ in an $A$--region across $C_2$ to an $A$--region, or $B$--region to a $B$--region, it will be type~$\mathcal{D}$, elementary, as in Figure~\ref{fig:EPDformD}, left.  

If $\zeta_2$ does not run upstream from $C_W$,
Lemma~\ref{lemma:AdjacentPoints} implies that the diagram has the form of one of the two hanging regions on the left of Figure \ref{fig:lollipops}.
\end{proof}

The previous results imply that if there is not a 2--edge loop, $\beta_V$ and $\beta_W$ lie in different polyhedral regions and the regions in which they lie have one of the forms of Figure~\ref{fig:lollipops}.  We set up some notation.  For each of the four forms of Figure~\ref{fig:lollipops}, there is a horizontal state circle at the bottom of the diagram.  In the case we are looking at the region of $\beta_W$, call that state circle $C_W$.  If we are looking at $\beta_V$, call that state circle $C_V$.  Note that $C_W$ might equal $C_V$, but we can assume that the regions of $\beta_W$ and $\beta_V$ are distinct by Lemma~\ref{lemma:SameRegion}.

In the following lemmas, we use the forms of Figure~\ref{fig:lollipops} to try to determine the form of adjacent regions, and eventually show there must be a 2--edge loop.  The next two lemmas are the generalization of \cite[Lemmas~6.11, 6.12]{fkp} to the homogeneously adequate case.

\begin{lemma}\label{lemma:StaircaseBetween}
Suppose $C_W \neq C_V$ and that $\zeta_2$, directed away from $\beta_W$, crosses $C_W$ running downstream.  Then for each $\zeta_i$, $i=1, 2$, there exists a staircase whose top is on $C_W$, with at least one segment, and with the following additional properties:
\begin{enumerate}
\item The arc $\zeta_i$ runs adjacent to each segment of the staircase it defines.
\item Both staircases run between $C_W$ and a state circle $C$, where either $C$ bounds a non-prime half--disk $D$ if $D$ is the first non-prime half--disk that $\zeta_2$ runs into without exiting, or $C=C_V$ if $\zeta_2$ exits every non-prime half--disk it enters.
\item Although the staircases of $\zeta_1$ and $\zeta_2$ may have distinct segments, the segments connect the same pairs of state circles each step of the way.
\end{enumerate}
\end{lemma}

\begin{proof}
To obtain the staircase of $\zeta_2$, apply Lemma~\ref{lemma:downstream} (Downstream) to $\zeta_2$ directed away from $\beta_W$.  Since $\zeta_2$ runs downstream across $C_W$, there is at least one segment of the resulting staircase, $\zeta_2$ runs adjacent to each segment, and the staircase either reaches $C_V$, or ends on the first state circle $C$ for which $\zeta_2$ crosses into a non-prime half--disk adjacent to $C$ and does not exit.

Now consider the arc $\zeta_1$.  If $\zeta_2$ crosses into a non-prime half--disk without exiting, then the state circle $C$ bounding that half--disk does not separate $\beta_V$ and $\beta_W$.  Hence note in this case that $\zeta_1$ must also cross into the non-prime half--disk, and it must do so by crossing $C$.  Since $C$ is not separating, $\zeta_1$ must actually cross it twice, once upstream and once downstream, by Lemma~\ref{lemma:utility} (Utility).  If, however, the staircase of $\zeta_2$ runs to $\beta_V$ without running into a non-prime half--disk separating $\beta_V$ and $\beta_W$, then $\zeta_2$ crosses $C_V$ running downstream.  It follows from \cite[Lemma~4.16]{fkp} that $\zeta_1$ must cross $C_V$ running upstream.  In all cases, $\zeta_1$, directed from $\beta_W$ to $\beta_V$, crosses the last state circle in the staircase of $\zeta_2$ (hereafter referred to as $C$) in the upstream direction.

To obtain the staircase of $\zeta_1$, change the direction on $\zeta_1$ to run downstream from $C$ to $C_W$ and apply Lemma~\ref{lemma:downstream} (Downstream).  This finishes the first two claims of the lemma.

To prove the final claim, the staircases of $\zeta_1$, $\zeta_2$, and arcs between their tops on $C_W$ and on $C$ define a closed curve $\gamma$ in the graph $H_\sigma$.  This curve crosses a state circle between every pair of segments, hence $\gamma$ must cross each such state circle twice, and so the state circles connect in pairs within the region bounded by $\gamma$.  By $\sigma$--adequacy of the diagram, no state circle can connect to the next adjacent state circle in the staircase.  This forces the two staircases to have the same number of stairs, and the state circles to connect directly left to right, giving the third claim of the lemma.
\end{proof}

\begin{lemma}[Full staircase]\label{lemma:FullStaircase}
Suppose $C_W \neq C_V$, and that $\zeta_2$ crosses $C_W$ running downstream when directed from $\beta_W$ to $\beta_V$. Then either:
\begin{enumerate}
\item the EPD runs over a 2--edge loop of one of the forms of Proposition~\ref{prop:EPDTypes}; or
\item $\beta_W$ is one of four forms:
\begin{enumerate}
\item\label{item:A-rainbow} as on the far left of Figure~\ref{fig:lollipops} (hanging--$A$), with a nonempty sequence of segments forming a right--down staircase below $C_W$ (i.e.\ all segments in $A$--regions), and with $\zeta_1$ and $\zeta_2$ adjacent to either side of each segment and each state circle of that staircase,
\item\label{item:B-rainbow} as on second to left of Figure~\ref{fig:lollipops} (hanging--$B$) with a nonempty sequence of segments forming a left--down staircase below $C_W$ (i.e.\ all segments in $B$--regions), and with $\zeta_1$ and $\zeta_2$ adjacent to either side of each segment and each state circle of the staircase,
\item\label{item:A-lollipop} as on the second to right of Figure~\ref{fig:lollipops} (standing--$A$), with a nonempty left--down staircase ($B$--staircase), and with $\zeta_1$ and $\zeta_2$ adjacent to either side,
\item\label{item:B-lollipop} as on the far right in Figure~\ref{fig:lollipops} (standing--$B$) with a nonempty right--down staircase ($A$--staircase) with $\zeta_1$ and $\zeta_2$ adjacent on either side.
\end{enumerate}
\end{enumerate}
In the second case, the staircase is maximal in the sense that it either runs to $C_V$, or runs to the first state circle $C$ for which there is a non-prime half--disk on $C$ separating $\beta_W$ and $\beta_V$.  
\end{lemma}

The options are shown in Figure~\ref{fig:staircases}.

\begin{figure}
  \begin{tabular}{cccc}
    \includegraphics{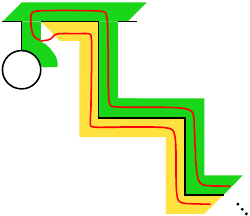} &
    \includegraphics{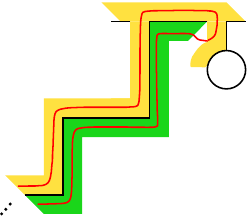} &
    \includegraphics{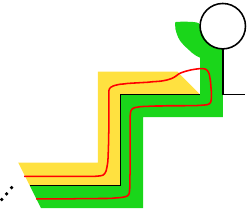} &
    \includegraphics{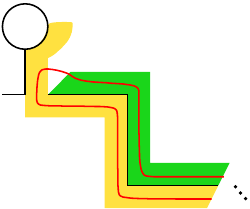} \\
    \eqref{item:A-rainbow} & \eqref{item:B-rainbow} & \eqref{item:A-lollipop} & \eqref{item:B-lollipop}
  \end{tabular}
  \caption{Forms of staircases from Lemma~\ref{lemma:FullStaircase}.  We say the right--down staircases on either end are $A$--staircases, and the left--down staircases in the center are $B$--staircases.}
  \label{fig:staircases}
\end{figure}

\begin{proof}
By Lemma~\ref{lemma:StaircaseBetween}, there exists a staircase from $C_W$ either to $C$ or $C_V$.  Let $k$ be the number of stairs traversed from $C_W$ in the staircase.  We will show by induction on $k$ that either there is a 2--edge loop of one of the proper forms, or the $(k+1)$-st stair has the form of case~(2) of the statement of the lemma. 

For the base case of the induction, note that if we have traversed zero stairs of the staircase, that $\beta_W$ must have one of the four forms of Figure~\ref{fig:lollipops} by Lemma~\ref{lemma:lollipops}.  We may think of these as having one of the forms in Figure~\ref{fig:staircases} when there are zero steps.  (Note Lemma~\ref{lemma:StaircaseBetween} will imply $k>0$.)

So suppose after traversing $k\geq 0$ stairs there is no 2--edge loop, and $\beta_W$ has one of the four forms in the statement of the lemma, possibly with an empty staircase if $k=0$.  Let $s_1$ denote the $(k+1)$-st segment of the staircase of $\zeta_1$, and let $s_2$ denote the $(k+1)$-st segment of the staircase of $\zeta_2$.  If $s_1 \neq s_2$, then there is a 2--edge loop at this step of the staircases.  Because the boundary of the EPD runs through tentacles of different color adjacent to the segments of the 2--edge loop, the 2--edge loop is of type~$\mathcal{A}$ or $\mathcal{C}$.  If we have an $A$--staircase running to an $A$--region, or a $B$--staircase running to a $B$--region, it will be type~$\mathcal{A}$, general.  Otherwise, it will be type~$\mathcal{C}$, general.

So suppose $s_1 = s_2$.  Now, the bottom of the $k$-th segment meets $C_k$ at a point $p$ and $\zeta_1$ and $\zeta_2$ both run adjacent to $p$ on $C_k$ (or if $k=0$ and $\beta_W$ lies in a standing region, then $p$ is the point at the end of the segment on $C_W$, and $\zeta_1$ and $\zeta_2$ both run adjacent to $p$).  The top of the $(k+1)$-st segment meets $C_k$ at a point $p'$, and $\zeta_1$ and $\zeta_2$ both run adjacent to $p'$.  If one of $\zeta_1$, $\zeta_2$ runs upstream, then Lemma~\ref{lemma:AdjacentLoop} (Adjacent Loop) implies that this arc runs through a 2--edge loop.  If $\zeta_2$ runs through the 2--edge loop, then the loop is on the same side of $C_k$ as the staircase from $\beta_W$, and it will be of type~$\mathcal{F}$ if $k>0$ or $\beta_W$ lies in a hanging region, and of type~$\mathcal{B}$--elementary if $k=0$ and $\beta_W$ lies in a standing region.  If $\zeta_1$ runs through the 2--edge loop, then the 2--edge loop lies on the opposite side of $C_k$ as the staircase from $\beta_W$, and it will be of type~$\mathcal{G}$.

Suppose then that there is no 2--edge loop.  Lemma~\ref{lemma:AdjacentPoints} (Adjacent Points) implies that $\zeta_1$ and $\zeta_2$ run parallel to $C_k$ between $p$ and $p'$. 

Now, if $\beta_W$ is in a hanging--$A$ region, by induction the $k$ steps between $C_W$ and $C_k$ form a right--down staircase through $A$--regions (case (2)\eqref{item:A-rainbow}).  Arcs $\zeta_1$ and $\zeta_2$ run parallel to $C_k$ from the end of the $k$-th segment to the beginning of the $(k+1)$-st.  The arc $\zeta_2$, in the green face, is running to the right across the top of $C_k$.  If the $(k+1)$-st segment of the staircase is in a $B$--region, then the arc $\zeta_1$, in the orange face, must run to the left across the bottom of $C_k$.  This contradicts the fact that the arcs run parallel.  So the $(k+1)$-st segment is in an $A$--region, and $\zeta_1$, $\zeta_2$ run parallel to each other from the base of the $k$-th step to the base of the $(k+1)$-st step.  Thus the staircase is still of the form in case (2)\eqref{item:A-rainbow}.  The same argument, up to reflection, applies if $\beta_W$ lies in a hanging--$B$ region. 

Now suppose that $\beta_W$ lies in a standing--$A$ region, and the staircase from $C_W$ to $C_k$ is of the form in case (2)\eqref{item:A-lollipop}.  That is, each segment runs through a $B$ region.  Again if there is no 2--edge loop, then $\zeta_1$ and $\zeta_2$ run parallel to $C_k$ from the end of the $k$-th segment to the beginning of the $(k+1)$-st segment.  The arc $\zeta_2$, in the orange face, runs from the bottom of the $k$-th segment to the left across the top of $C_k$.  Thus the arc $\zeta_1$, in the green face, must run to the left across the bottom of $C_k$.  This will happen only if the region containing the $(k+1)$-st segment is a $B$--region, giving the inductive step in this case.  The same argument, up to reflection, gives the inductive step in the case $\beta_W$ lies in a standing--$B$ region.  
\end{proof}

\begin{lemma}[Terminating staircase]\label{lemma:TermStair}
Suppose $\zeta_2$ runs across $C_W$ in the downstream direction when directed from $\beta_W$ to $\beta_V$, out of every non-prime half--disk that it enters, and terminates with $\zeta_2$ crossing $C_V$. Then the conclusion of Proposition~\ref{prop:EPDTypes} holds.
\end{lemma}

\begin{proof}
By Lemma~\ref{lemma:FullStaircase}, we may assume that either the result holds, or $C_V=C_W$, or $\beta_W$ has one of the four forms of Lemma~\ref{lemma:FullStaircase} (Staircase).  Swapping the roles of $V$ and $W$, in the latter case the lemma implies that $\beta_V$ also has one of those four forms.

Suppose first that $C_V \neq C_W$ and that $\beta_W$ lies in a hanging--$A$ region.  Then there is a nonempty right--down staircase from $\beta_W$ to $\beta_V$ with $\zeta_1$ and $\zeta_2$ adjacent to each segment and each state circle; segments run only through $A$--regions, as in \eqref{item:A-rainbow} of Figure~\ref{fig:staircases}.  Similarly, Lemma~\ref{lemma:FullStaircase} implies there is such a staircase from $\beta_V$ to $\beta_W$.  This must agree with the former staircase, so each segment in it must run only through $A$--regions.  So $\beta_V$ is either in a hanging--$A$ or standing--$B$ region, and $\beta_V$ is attached to a right--down staircase.  However, note that in these two cases, the green (darker shaded) face will be adjacent to the side of $C_V$ that contains the staircase from $\beta_W$, since green runs along ``tops'' of state circles, i.e.\ the side containing $\beta_W$, on the right--down staircase from $\beta_W$ in a hanging--$A$ region, item \eqref{item:A-rainbow} of Figure~\ref{fig:staircases}.  The green face will also be adjacent to the opposite side of $C_V$, since if $\beta_V$ lies in a hanging--$A$ or standing--$B$ region, the green face is adjacent to the side of the state circle $C_V$ containing $\beta_V$.  But then we may draw an arc in the green face with one endpoint on one side of $C_V$, and one endpoint on the other, which is disjoint from $C_V$ in the interior.  This arc can be connected in a neighborhood of $C_V$ to a closed curve in $H_\sigma$ meeting $C_V$ exactly once, which is a contradiction.  The same argument, reflected, implies that $\beta_W$ does not lie in a hanging--$B$ region if $C_W \neq C_V$.

So now suppose $C_V\neq C_W$ and $\beta_W$ lies in a standing--$A$ region.  Then $\zeta_1$ and $\zeta_2$ run adjacent to a nontrivial left--down staircase through $B$--regions.  Again, $\beta_V$ must also give a staircase through $B$ regions, so $\beta_V$ either lies in a hanging--$B$ or standing--$A$ region.  Again in either case we have orange faces adjacent to either side of $C_V$, giving a contradiction as above.  A reflection of this argument implies that $\beta_W$ does not lie in a standing--$B$ region if $C_V\neq C_W$.

Thus $C_V = C_W$.  We still know that $\beta_V$ and $\beta_W$ are either in hanging or standing regions, one on either side of $C:=C_W=C_V$.  To avoid having green or orange faces adjacent to both sides of $C$, we must have one of the combinations: hanging--$A$ with standing--$A$, hanging--$A$ with hanging--$B$, standing--$A$ with standing--$B$, and hanging--$B$ with standing--$B$.

In all cases, consider the point on $C$ where the segment of a hanging or standing region meets $C$.  In the hanging case, both $\zeta_1$ and $\zeta_2$ run adjacent to this point.  In the standing case, slide the point slightly to lie between $\zeta_1$ and $\zeta_2$ on $C$, and again we have $\zeta_1$ and $\zeta_2$ running adjacent to this point.
Then in all possible combinations, Lemma~\ref{lemma:AdjacentPoints} (Adjacent Points) implies that either there is a 2--edge loop, or $\zeta_1$ and $\zeta_2$ run parallel to $C$ between these points.  If there is a 2--edge on the same side of $C$ as a hanging region, then it must be of type~$\mathcal{F}$.  If there is a 2--edge loop on the same side of $C$ as a standing region, it must be of type~$\mathcal{B}$--elementary.  

So suppose that $\zeta_1$ and $\zeta_2$ run parallel to $C$ between the two points.  Consider the hanging--$A$ with standing--$A$ case, sketched with $C$ horizontal, with the hanging region on top.  Note that at the base of the segment on $C$ in the hanging region, the arc in the green face runs to the right adjacent to $C$.  However, on the opposite side, the arc in the orange face runs to the left.  This is a contradiction.  We obtain the same contradiction for hanging--$A$ with hanging--$B$, standing--$A$ with standing--$B$, and for hanging--$B$ with standing--$B$.  So in all cases, there must be a 2--edge loop.
\end{proof}

\begin{lemma}[Inside non-prime arcs]\label{lemma:InNonPrime}
Suppose $\beta_W$ and $\beta_V$ are separated by a non-prime arc $\alpha$, with the arc $\zeta_2$, say, crossing $\alpha$. Suppose $\alpha$ is outermost among all such arcs, with respect to $\beta_W$. That is, $\alpha$ is the first such non-prime arc crossed by $\zeta_2$ when directed toward $\beta_V$. Then we have the conclusion of Proposition~\ref{prop:EPDTypes}.
\end{lemma}

\begin{proof}
As in the proof of \cite[Lemma 6.14]{fkp}, we break into two cases: whether $\zeta_2$ runs upstream after crossing $\alpha$ or not.

\underline{Case 1}.  Suppose $\zeta_2$ does not run upstream after crossing $\alpha$.  Let $C$ be the state circle on which the endpoints of $\alpha$ lie.  First, note that $\zeta_2$ cannot run downstream across $C$ after crossing $\alpha$, for since $\beta_V$ lies inside the non-prime half--disk bounded by $\alpha$ and $C$, if $\zeta_2$ crosses $C$ running downstream away from $\beta_V$, it must cross $C$ again to reach $\beta_V$.  This contradicts the Utility Lemma~\ref{lemma:utility}.

Since $\zeta_2$ does not run upstream or downstream after crossing $\alpha$, it must run to the vertex $\beta_V$.  We know that $\beta_V$ has one of the four forms in Figure \ref{fig:lollipops}. Since $\zeta_2$ does not run upstream or downstream, but both arcs in the hanging regions run upstream or downstream, $\beta_V$ must be as on the right of that figure, one of the two standing regions.  We will assume it is a standing--$A$ region, since the standing--$B$ argument is the same up to reflection.  Thus $\zeta_2$ lies in an orange face. 

Lemma~\ref{lemma:FullStaircase} (Full staircase) implies that on the opposite side of $\alpha$, either there is a 2--edge loop, or $\zeta_1$ and $\zeta_2$ run adjacent to a (possibly empty) staircase of $\beta_W$.  If the staircase of $\beta_W$ is non-empty, then it has one of the four forms given by that lemma.  To ensure the colors of faces match correctly, $\zeta_2$ must run through an orange face, and so the staircases must either be type~\eqref{item:A-lollipop}, standing--$A$ and left--down, or type~\eqref{item:B-rainbow}, hanging--$B$ and left--down.  But now note that a left--down staircase terminates in a $B$--region, at the state circle $C$, whereas we are assuming $\alpha$ has endpoints in an $A$--region, on the same side of the state circle $C$.  This contradicts our assumption of $\sigma$--homogeneity.  

Thus we reduce to the case that the staircase of $\beta_W$ is empty, or consists of no segments.  The same argument as above implies that $\beta_W$ does not lie in a hanging region.  Thus $\beta_W$ also lies in a standing region, and to ensure homogeneity, it must be a standing--$A$ region.  We assume $\zeta_2$ meets no state circles when running from $\beta_W$ to $\beta_V$, else we'll get a 2--edge loop (of type~$\mathcal{B}$--elementary) by Lemma~\ref{lemma:AdjacentLoop} (Adjacent loop).  Additionally, $\zeta_1$ crosses $C$ twice, and we may assume it meets no state circles between crossings, or again Lemma~\ref{lemma:AdjacentLoop} (Adjacent loop) implies that there is a 2--edge loop, this time of type~$\mathcal{D}$--elementary or $\mathcal{B}$--elementary.  Then $\zeta_1 \cup \zeta_2$ gives a closed curve in the graph $H_\sigma$ meeting $H_\sigma$ exactly twice.  This gives a closed curve meeting the diagram twice with crossings on either side (inside and outside the half--disk bounded by $\alpha$), contradicting the fact that the diagram is prime, Definition~\ref{def:prime}.

\underline{Case 2}.  The arc $\zeta_2$ runs upstream after crossing $\alpha$.

If $\zeta_2$ runs back to $C$ afterward, then there must be a 2--edge loop by Lemma~\ref{lemma:AdjacentLoop} (Adjacent loop), of type $\mathcal{F}$ in Figure~\ref{fig:EPDformF}.  So suppose $\zeta_2$ does not return to $C$ after running upstream, across a state circle we will label $C'$. 

We claim that $\zeta_1$ must also cross $C'$. 
If $C'$ separates $\beta_V$ and $\beta_W$, then both $\zeta_1$ and $\zeta_2$ must cross $C'$.  So suppose $C'$ does not separate $\beta_V$ and $\beta_W$.  Then $\zeta_2$ must cross $C'$ twice, the second time in the downstream direction, running toward $\beta_V$.  By assumption, $\zeta_2$ does not run back to $C$, hence it runs downstream to some $C''$.  The arc $\zeta_1$ crosses $C$ running downstream toward $\beta_V$, so it must also run downstream to $C''$, following tentacles to the same side of $C''$ as $\zeta_2$.  But this leads to a contradiction: such arcs cannot meet in a vertex, but will be separated by segments.
So we conclude that both $\zeta_1$ and $\zeta_2$ run along segments between $C$ and $C'$.
If these segments are distinct, then there is a 2--edge loop, of type~$\mathcal{E}$. 

If these segments are not distinct, then both $\zeta_1$ and $\zeta_2$ run adjacent to the endpoint $p$ of the segment on $C$.  They also both run adjacent to a point $p'$ on $C$ where an endpoint of $\alpha$ meets $C$.  Between those two points, $\zeta_2$ runs over $\alpha$, hence meets no segments or state circles.  If the arc $\zeta_1$ meets state circles between, then Lemma~\ref{lemma:AdjacentLoop} (Adjacent loop) implies there is a 2--edge loop.  This will either be of type~$\mathcal{B}$--general, of type~$\mathcal{D}$--general, or of type~$\mathcal{G}$, depending on whether $\zeta_1$ runs into an $A$ or $B$--region, and whether it runs over a non-prime arc.

If $\zeta_1$ meets no state circles in between $p$ and $p'$, then take the portions of $\zeta_1$ and $\zeta_2$ between these two points and connect their ends. This gives a loop meeting the diagram exactly twice, with crossings on either side, contradicting the fact that the diagram is prime.
\end{proof}

\begin{proof}[Proof of Proposition~\ref{prop:EPDTypes}]
Let $E$ be an essential product disk embedded in $P_\sigma$.  By Lemma~\ref{lemma:EPDtoSquare}, there exists a normal square $E'$ with each white side cutting off a triangle in such a way that in counter-clockwise order, edges of the triangle are colored white--orange--green.  As usual, label the white sides $\beta_V$ and $\beta_W$.  By Lemma~\ref{lemma:SameRegion}, if $\beta_V$ and $\beta_W$ are in the same polyhedral region, then there is a 2--edge loop in the upper polyhedron with one of the desired forms.  If $\beta_V$ and $\beta_W$ are not in the same region, then the sides of $\partial E'$ in the shaded faces, which we label $\zeta_1$ and $\zeta_2$, must run over a sequence of state circles to get from $\beta_W$ to $\beta_V$.  Lemma~\ref{lemma:TermStair} implies that if $\zeta_2$ crosses $C_W$ and runs out of every non-prime half disk until it crosses $C_V$, then the boundary of $E'$ runs over tentacles adjacent to segments of a 2--edge loop as desired.  If instead $\zeta_2$ runs into a non-prime half--disk without exiting, then Lemma~\ref{lemma:InNonPrime} implies that the boundary of $E'$ runs over tentacles adjacent to a 2--edge loop as desired.
\end{proof}

The proof of Theorem~\ref{thm:main} is now immediate.

\begin{proof}[Proof of Theorem~\ref{thm:main}]
Given an EPD $E$, pull it into a normal square $E'$ as in Lemma~\ref{lemma:EPDtoSquare}.  By Proposition~\ref{prop:EPDTypes}, the normal square has one of the forms of Figures~\ref{fig:EPDformA} through \ref{fig:EPDformG}.  In every case, the boundary of the square $E'$ runs through tentacles adjacent to segments of a 2--edge loop. 
We may slide the white sides $\beta_V$ and $\beta_W$ of $\partial E'$ back onto the ideal vertices to recover an EPD isotopic to $E$, and we may do this in such a way that $\partial E$ runs over the same 2--edge loop as $\partial E'$, concluding the proof of the theorem.
\end{proof}

\section{Volume applications}\label{sec:applications}

In this section, we apply of Theorem~\ref{thm:main} to bound volumes of classes of links in terms of their diagrams.  
First, we need some definitions.

\begin{define}\label{def:Twist}
A \emph{twist region} is a string of bigon regions in the complement of the diagram graph that are adjacent end to end, or a single crossing adjacent to no bigons.  We also require the diagram to be alternating in a twist region, else a Reidemeister move simplifies the diagram, and we require that a twist region be maximal in the sense that there are no additional bigons adjacent on either side.  An example is shown in Figure~\ref{fig:TwistShort}, left.
\end{define}

\begin{figure}
  \includegraphics{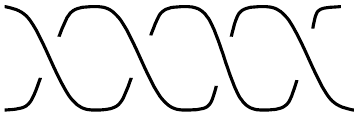} \hspace{.1in}
  \includegraphics[width=1.4in]{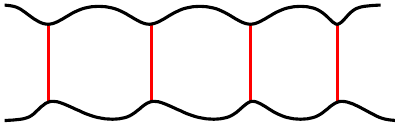} \hspace{.1in}
  \includegraphics[width=1.4in]{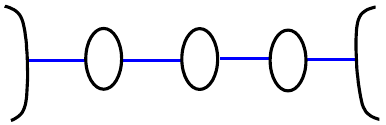}
  \caption{Left to right:  A twist region, its short resolution, and its long resolution.}
  \label{fig:TwistShort}
\end{figure}

Given a twist region with more than one crossing, an adequate state will choose either all $A$ or all $B$ resolutions for that twist region.  One choice will produce a portion of $H_\sigma$ consisting of two state circles, with segments parallel between them.  The other choice will produce a sequence of state circles from the interior of the twist region, each meeting only two segments.  We call the former the \emph{short} resolution, and the latter the \emph{long} resolution.  See Figure \ref{fig:TwistShort}.

The following definition is as in \cite{fkp:hyplinks}.

\begin{define}\label{def:2edgeloop}
Two edges of a 2--edge loop in $\GG_\sigma$ \emph{belong to the same twist region} $R$ of the diagram if the edges correspond to crossings of $R$, and the resolution of $R$ is the short one.  We say that $\GG_\sigma$ satisfies the \emph{2--edge loop condition} if every 2--edge loop in $\GG_\sigma$ belongs to the same twist region.  In this case, we also say that the diagram satisfies the 2--edge loop condition for $\sigma$.
\end{define}

Theorem~\ref{thm:main} gives us the following corollary, analagous to \cite[Corollary~5.19]{fkp}.

\begin{corollary}\label{cor:2edgeloop}
Suppose that $D(K)$ is a prime diagram with homogeneous, adequate state $\sigma$, and suppose that $\GG_\sigma$ satisfies the 2--edge loop condition.  Then $||E_c||=0$.
\end{corollary}

\begin{proof}
Recall from Definition~\ref{def:Ec} that $||E_c||$ is the number of complex EPDs in a minimal spanning set for the $I$--bundle of $S^3\cut S_\sigma$, from \cite[Lemma~5.8]{fkp}. Let $R$ be a twist region of $D:=D(K)$ whose resolution under $\sigma$ is the short resolution.  
In \cite[Lemma~5.17]{fkp}, which extends immediately to the homogeneously adequate case (see Section~5.6 of that paper), it is shown that if $\hat{D}$ is the diagram obtained by removing one crossing from $R$, then $E_c(D)$ and $E_c(\hat{D})$ have the same cardinality.

Now let $\bar{D}$ be the diagram obtained from $D$ by removing all bigons in the short resolutions of twist regions of $D$.  By induction, $||E_c(D)|| = ||E_c(\bar{D})||$. Moreover, because $\GG_\sigma$ satisfies the 2--edge loop condition, this will remove all 2--edge loops from $\GG_\sigma(\bar{D})$. 

But by Theorem~\ref{thm:main}, any EPD coming from $\bar{D}$ has boundary running over a 2--edge loop.  Since $\GG_\sigma(\bar{D})$ has no 2--edge loops, there must be no EPDs in $E_c(\bar{D})$.  Hence $||E_c||=0$.
\end{proof}


Suppose that $M$ is a 3--manifold and $S$ is an essential surface embedded in $M$.  Recall that the \emph{guts} of $M\cut S$ is defined to be the complement of the maximal $I$--bundle in $M\cut S$.  Work of Agol \cite{agol:guts}, extended by Kuessner \cite{kuessner:guts}, says that guts can be used to bound the Gromov norm of $M$.  In particular,
\begin{equation}\label{eqn:GromovBound}
||M|| \geq 2\,\chi_-(\rm{guts}(M\cut S)).
\end{equation}
This was extended by Agol, Storm, and Thurston to give bounds on the volumes of hyperbolic manifolds \cite{ast}.  If $M$ is hyperbolic, then
\begin{equation}\label{eqn:AST}
\vol(M) \geq v_8\,\chi_-(\rm{guts}(M\cut S)),
\end{equation}
where $v_8 = 3.6638\dots$ is the volume of a regular ideal octahedron.

In \cite{fkp}, it was shown that
\[\chi_-(\rm{guts}(S^3\cut S_\sigma)) = \chi_-(\GG'_\sigma) - ||E_c||,\]
where recall $\GG'_\sigma$ is the reduced state graph of Definition~\ref{def:StateGraph}.  
This is Theorem~5.14 of that paper, extended to homogeneously adequate links in \cite[Section~5.6]{fkp}.  Combining this with Corollary~\ref{cor:2edgeloop} above, we obtain the following corollary, analogous to \cite[Corollary~9.4]{fkp}.

\begin{corollary}\label{cor:2loopVolumes}
Suppose that $D(K)$ is a prime diagram of a hyperbolic link $K$, with homogeneous, adequate state $\sigma$, and suppose that $\GG_\sigma$ satisfies the 2--edge loop condition.  Then
\[ \vol(S^3-K) \geq v_8\,(\chi_-(\GG_\sigma')).\]
\end{corollary}

\begin{proof}
By Theorem~5.14 of \cite{fkp}, which applies to the homogeneously adequate case, $\chi_-(\guts(S^3\cut S_\sigma)) = \chi_-(\mathbb{G}'_\sigma) - ||E_c||$ where $||E_c||$ is the number of complex essential product disks in a spanning set for the $I$--bundle.  By Corollary~\ref{cor:2edgeloop} above, $||E_c||=0$.
Now equation \eqref{eqn:AST} above implies that the volume of $S^3-K$ is bounded below by $v_8\,\chi_-(\mathbb{G}'_\sigma)$.
\end{proof}



\bibliographystyle{amsplain}

\bibliography{biblio}

\end{document}